\documentclass[final, 11pt,reqno]{amsart}
\usepackage{amssymb,amsmath,graphicx,amsfonts,euscript}
\usepackage{color}
\usepackage[pagewise]{lineno}
\usepackage{showkeys}
\usepackage{fancyhdr}
\usepackage[margin=1in]{geometry}

\newtheorem{theorem}{\bf Theorem}[section]

\newtheorem{corollary}[theorem]{\bf Corollary}
\newtheorem{proposition}[theorem]{\bf Proposition}
\newtheorem{define}[theorem]{\bf Definition}
\newtheorem{remark}{\bf Remark}

\newtheorem{lemma}[theorem]{\bf Lemma}

\newtheorem{assumption}[theorem]{Assumption}

\newcommand{\beq}{\begin{equation}}
\newcommand{\eeq}{\end{equation}}
\newcommand{\ben}{\begin{eqnarray}}
\newcommand{\een}{\end{eqnarray}}
\newcommand{\beno}{\begin{eqnarray*}}
\newcommand{\eeno}{\end{eqnarray*}}

\numberwithin{equation}{section}
\subjclass[2010]{35A01, 35B45, 35R11, 35Q92.}
\keywords{Keller-Segel system, mixing, degenerate dissipation, weakening nonlinear.}
\makeatother
\title[Keller-Segel system with degenerate dissipation and mixing]{Global well-posedness for a generalized Keller-Segel system with degenerate dissipation and mixing}
\author[B.Shi]{Binbin Shi}
\author[W.Wang]{Weike Wang}
\address{School of Mathematical Sciences , Shanghai Jiao Tong University, Shanghai, 200240,  P.R.China.}
\email{shibbing@163.com,\ binbinshi@sjtu.edu.cn}
\address{School of Mathematical Sciences  and Institute of Natural Science, Shanghai Jiao Tong University, Shanghai, 200240,  P.R.China.}
\email{wkwang@sjtu.edu.cn}

\begin{document}

\begin{abstract}
 We study the mixing effect for a generalized Keller-Segel system with degenerate dissipation and advection by a weakly mixing. Here the attractive operator has weak singularity, namely, the negative derivative appears in the nonlinear term by singular integral. Without advection, the solution of equation blows up in finite time. We show that the global well-posedness of solution with large advection. Since dissipation term degenerate into the damping, the enhanced dissipation effect of mixing no longer occurs, we prove that the mixing effect can weak the influence of nonlinear term. In this case, the mixing effect is similar with inviscid damping of shear flow. Combining to the mixing effect and damping effect of degenerate dissipation, the global $L^\infty$ estimate of solution is established.
\end{abstract}

\maketitle

\section{Introduction}
The generalized parabolic-elliptic Keller-Segel system on torus $\mathbb{T}^d$ with fractional dissipation in the presence of an incompressible flow is as follows
\begin{equation}\label{eq:1.1}
\begin{cases}
\partial_t\rho+Au\cdot \nabla \rho+(-\Delta)^{\frac{\alpha}{2}}\rho+\nabla\cdot(\rho B(\rho))=0,\qquad  & x\in \mathbb{T}^d,\ t>0,\\
\rho(0,x)=\rho_0(x),  & x\in \mathbb{T}^d.
\end{cases}
\end{equation}
Here the quantity $\rho(t,x)$ denotes the density of microorganisms, it is a real-valued function of $t$ and $x$.  $u(x)$ is a divergence-free vector field which is an ambient flow, and $A$ is a positive constant. The domain $\mathbb{T}^d$ is the periodic box with dimension $d\geq2$. The $0\leq\alpha\leq2$ is parameter controlling the strength of the dissipative term. The linear vector operator $B$ is called attractive operator, which could be represented as
\begin{equation}\label{eq:1.2}
B(\rho)=\nabla K\ast \rho, \ \ \ \ x\in \mathbb{T}^d,
\end{equation}
where $K$ is a periodic convolution kernel, which is smooth away from the origin, and near $x=0$
\begin{equation}\label{eq:1.3}
\nabla K\sim -\frac{x}{|x|^\beta}, \ \  \beta\in [2,d+1).
\end{equation}
The similar discussion can be referred to (see \cite{Hopf.2018, Shi.2021, Shi.2019}). Notice that, for the whole space, attractive operator $B$ can be written as
\begin{equation}\label{eq:1.4}
B(\rho)=\nabla((-\Delta)^{-\frac{d+2-\beta}{2}}\rho), \ \ \  x\in \mathbb{R}^d.
\end{equation}
If $\alpha=0$, the fractional dissipation is called degenerate dissipation, and fractional dissipation become damping. And we consider the attractive operator $B$ is weak singularity ($2\leq\beta<d$) in this paper.

\vskip .05in

In the absence of the advection, the Equation (\ref{eq:1.1}) is the generalized  Keller-Segel system with fractional dissipation
\begin{equation}\label{eq:1.5}
\partial_t\rho+(-\Delta)^{\frac{\alpha}{2}}\rho+\nabla\cdot(\rho B(\rho))=0,\quad \rho(0,x)=\rho_0(x),\quad x\in \Omega,
\end{equation}
where $\Omega$ is $\mathbb{R}^d$ or $\mathbb{T}^d$. For $d\geq2$, the solution of Equation (\ref{eq:1.5}) may blow up in finite time for large initial data (see \cite{Biler.2010, Blanchet.2006, Corrias.2004,Hopf.2018, Kiselev.2016, Lafleche.2019, Li.2010,Nagai.1995, Senba.2002}). In particular, the solution of Equation (\ref{eq:1.5}) also may blow up in finite time in the case of $\alpha=0, 2\leq\beta<d$  (see \cite{Lafleche.2019}).

\vskip .05in

Recently, the mixing effect of advection have been widely studied by many scholar. For the advection diffusion equation
\begin{equation}\label{eq:1.6}
\partial_tf+Au\cdot\nabla f-\Delta f=0,\ \ \ div u=0, \ \ \ f(0,x)=f_0(x), \ \ \ x\in \mathbb{T}^d.
\end{equation}
The enhanced dissipation effect of mixing is an interesting phenomenon, it means that the dissipation effect will be enhanced and the $L^2$ norm of solution to the Equation (\ref{eq:1.6}) has a faster decaying if $A$ is large enough. In fact, the enhanced dissipation effect does not occur for all incompressible flows. Next, two types of incompressible flow with dissipation enhancement effect are introduced. First, $u$ is shear flow, we can refer to \cite{Bedrossian.201701,Wei.2021}, and some specific shear flows are considered to the Naver-Stokes equation. For example, Coutte flow (see \cite{Bedrossian.201602,Masmoudi.2020}), Poiseuille flow (see \cite{Michele.2020}) and Kolmogorov flow (see \cite{Lin.2019, Wei.201901,Wei.2020}). Secondly,  Constantin, Kiselev, Ryzhik, and Zlato\v{s} (see \cite{Constantin.2008}) defined the relaxation enhancing flow, and proved that $u$ is relaxation enhancing flow if and only if $u\cdot \nabla$ has no nontrivial $\dot{H}^{1}$ eigenfunction. In Equation (\ref{eq:1.6}), if consider the fractional dissipation, Hopf and Ridrigo (see \cite{Hopf.2018}) defined the $\alpha$-relaxation enhancing flow. Some other studies and examples of relaxation enhancing flow can be also referred to \cite{Constantin.2008,Fayad.2006,Fayad.2002,Kiselev.2016}. In this paper, we consider a special relaxation enhancing flow--weakly mixing (see Definition \ref{def:2.1}), the more details are introduced in Section 2.

\vskip .05in

For the transport equation with mixing effect
\begin{equation}\label{eq:1.7}
\partial_tf+Au\cdot\nabla f=0,\ \ \ div u=0,\ \ \ f(0,x)=f_0(x), \ \ \ x\in \mathbb{T}^d.
\end{equation}
Since there is no dissipation term, the enhanced dissipation effect of $u$ no longer occurs in Equation (\ref{eq:1.7}), and $L^2$ norm of solution to Equation (\ref{eq:1.7}) can not describe the mixing effect of $u$. In fact, the mixing effect of $u$ is studied by $\dot{H}^{-1}$ norm of solution to Equation (\ref{eq:1.7}). This is another interesting phenomenon. For example, if $u$ is shear flow in $\mathbb{T}^2$, and $u\in C^{n_0+2}(\mathbb{T})$ such that $u'(y)=0$ in at most finitely many places and suppose that at each critical point, $u'$ vanishes to at most order $n_0\geq1$, then $k\neq 0$, we have
$$
\|\mathcal{P}_kf(t)\|_{\dot{H}^{-1}_y}\leq C\langle Akt\rangle^{-\frac{1}{1+n_0}} \|\mathcal{P}_kf(0)\|_{\dot{H}^{1}_y},
$$
where
$$
\langle Akt\rangle=\sqrt{1+( Akt)^2},
$$
and $\mathcal{P}_k$ denotes the projection to the $k$-th Fourier mode in $x$, the more details can be seen in \cite{Bedrossian.201701}. If $u$ is weakly mixing, we deduce by RAGE theorem (see Lemma \ref{lem:2.2}), for any $\delta,\tau>0$ and $A$ is large enough, if $f_0$ is a mean zero function, then
$$
\frac{1}{\tau}\int_{0}^\tau\|f(t)\|_{\dot{H}^{-1}}dt\leq \delta \|f_0\|_{L^2},
$$
the more details can be referred to \cite{Constantin.2008,Kiselev.2016}. In recent years, the inviscid damping phenomena of Euler equation around a shear flow attracted attention. For example, Coutte flow (see \cite{Bedrossian.201602,Ionescu.2020,Zeng.2011}), Kolmogorov flow (see \cite{Lin.2019,Wei.2020}) and other shear flows (see \cite{Bedrossian.201501,Jia.2020,Zhang.2018,Zhang.201902,Zillinger.2016,Zillinger.2017}). In fact, the inviscid damping is similar with mixing effect of Equation (\ref{eq:1.7}). If denote $v$ is velocity,  which is the solution of two dimensional incompressible equation. Consider the two dimensional linearized Euler vorticity equation around Coutte flow $(y,0)$
$$
\partial_t\omega+y\partial_x\omega=0,
$$
and we know that
$$
\|v\|_{L^2}\sim \|\omega\|_{\dot{H}^{-1}}.
$$

\vskip .05in

The nonlinear system with mixing effect have also been considered by many authors, and get some interesting phenomenons. For example, enhanced dissipation effect of mixing suppress blowing up of solution to Keller-Segel system have been widely studied by many scholars (see \cite{Bedrossian.2017,Hopf.2018,Kiselev.2016,Shi.2021,Shi.2019,Zi.2021}). The main idea is that the dissipation term become a large damping by enhanced dissipation effect of mixing, so the nonlinear effect can be controlled by large damping effect of dissipation term and obtain the fast decay. Another interesting question is that for the nonlinear system without viscous dissipation, we hope that the mixing effect still have a good influence for well-posedness. Obviously, the enhanced dissipation effect of mixing does not appear in nonlinear system, therefore, we need to look for new mechanism of mixing effect to suppress blowing up.
\vskip .05in

In this paper, we consider the Equation (\ref{eq:1.1}) when the viscous dissipation degenerate into damping ($\alpha=0$) and $2\leq \beta<d$, it is as follows
\begin{equation}\label{eq:1.8}
\begin{cases}
\partial_t\rho+Au\cdot \nabla \rho+\rho+\nabla\cdot(\rho B(\rho))=0,\ \ \ & x\in \mathbb{T}^d,\ t>0,\\
\rho(0,x)=\rho_0(x),\ \ \ & x\in \mathbb{T}^d,
\end{cases}
\end{equation}
where $u$ is weakly mixing. If $A=0$, the solution of Equation (\ref{eq:1.8}) may blow up in finite time (see \cite{Lafleche.2019}). We want to prove that mixing effect can suppress blowing up when $A$ is large enough. Since there is no  viscous dissipation, then the enhanced dissipation effect of mixing does not appear in Equation (\ref{eq:1.8}). We try to study the weakening nonlinear effect of mixing in Equation (\ref{eq:1.8}). Notice that, the attractive operator $B$ has weak singularity, combining to $2\leq \beta<d$ and (\ref{eq:1.2})-(\ref{eq:1.4}), the negative derivative of solution to Equation (\ref{eq:1.8}) appears in the nonlinear term  by singular integral. The later section of this paper will use negative derivative and mixing mechanism to weaken the nonlinear effect.

\vskip .05in

Our goal is to establish the global classical solution
of Equation (\ref{eq:1.8}) through mixing effect, its means that the blow up phenomenon  of solution of Equation (\ref{eq:1.5}) can be suppressed  in the case of $\alpha=0, 2\leq\beta<d$. For the Equation (\ref{eq:1.1}), in the previous work, the main idea is to use the enhanced dissipation effect of the mixing to improved the influence of the dissipative term, such that the blow up phenomena is suppressed. Here we consider mixing effect can weaken the influence of nonlinear term. In the case of $2\leq\beta<d$, the negative derivative of solution to Equation (\ref{eq:1.8}) appears in the nonlinear term. Then there is a negative-index Sobolev norm of solution to Equation (\ref{eq:1.8}) in the estimate of nonlinear term. Inspired by inviscid damping and the discussion of mixing effect of $u$ in Equation (\ref{eq:1.7}), we know that the negative-index Sobolev norm of solution to Equation (\ref{eq:1.8}) is small if $A$ is large enough. Thus, the strong nonlinearity of Equation (\ref{eq:1.8}) becomes weak nonlinearity. Since degenerate dissipation is damping, and we know that the damping can destroy weak nonlinear structure (see \cite{Chen.2016,Wang.2001}), then nonlinear term is controlled. Based on the above discussion , it is reasonable to consider global solution of Equation (\ref{eq:1.8}) in the case of $\alpha=0, 2\leq\beta<d$.

\vskip .05in

Now, let us state main result.

\begin{theorem}\label{thm:1.1}
Let $2\leq \beta<d, d>2$, for any initial data $\rho_0\geq0, \rho_0\in W^{3,\infty}(\mathbb{T}^d)$, there exists a positive constant $A_0=A(\beta,\rho_0, d)$, if $A\geq A_0$, then the unique classical solution $\rho(t,x)$ of Equation (\ref{eq:1.8}) is global in time, and one has
$$
\rho(t,x)\in C(\mathbb{R}^{+}; W^{3,\infty}(\mathbb{T}^d)).
$$
\end{theorem}

\begin{remark}\label{rem:2}
In fact, we obtained the global $L^p (1\leq p\leq\infty)$ estimate of solution to Equation (\ref{eq:1.8}) through mixing effect and damping effect of degenerate dissipation influence the nonlinear term. If we consider the Equation (\ref{eq:1.8}) without nonlinear term, it is as follows
\begin{equation}\label{eq:1.11}
\partial_t\eta+Au\cdot \nabla \eta+\eta=0.
\end{equation}
We know that the mixing effect has no any influence for establishing the $L^p$ estimate. If we consider the Equation (\ref{eq:1.8}) without full degenerate dissipation
\begin{equation}\label{eq:1.12}
\partial_t\rho+Au\cdot \nabla \rho+\nabla\cdot(\rho B(\rho))=0,
\end{equation}
the global solution of Equation (\ref{eq:1.12}) can not be obtained by mixing effect.
\end{remark}

The RAGE theorem (see \cite{Cycon.1987}) plays an important role in the analysis of mixing effect of weakly mixing, the details can be seen in \cite{Constantin.2008}, and is applied to consider the enhanced dissipation effect of mixing for the generalized Keller-Segel system  with relaxation enhancing flow (see \cite{Hopf.2018,Kiselev.2016,Shi.2019}). Recently, Wei \cite{Wei.2021} showed the enhanced dissipation of relaxation enhancing flow by the resolvent estimate of operator $-\Delta+Au\cdot \nabla$. However, the method of resolvent estimate is not suitable for the full degenerate dissipation case. In \cite{Hopf.2018,Kiselev.2016,Shi.2019}, based on the enhanced dissipation effect of mixing to the generalized Keller-Segel system  with weakly mixing, the global prior estimate was established  by isometric local extension if $A$ is large enough. We know that the constant $A$ is fixed by applying the the RAGE theorem, then we only to obtain the uniform time $T_c$. In \cite{Constantin.2008,Kiselev.2016,Shi.2019}, $T_c=T(N, \sigma, \mathcal{K}, U^t)$, where $N, \sigma$ and $U^t $ are fixed, $\mathcal{K}$ is a compact functional set. If $\alpha=0$, we can not find a compact functional set.

In this paper, we find that the RAGE theorem (see Lemma \ref{lem:2.3}) is valid for the case of uniformly bounded functional set of $L^2$ norm in the case of $\alpha=0$,  the details can see Remark \ref{rem:3}. Thus we can analysis of mixing effect in Equation (\ref{eq:1.8}) by the RAGE theorem (see Lemma \ref{lem:2.2}). Compared with enhanced dissipation, the weakening nonlinear effect of mixing is a new discovery of this paper. And we need some new technical to analysis the weakening nonlinear effect of mixing in Equation (\ref{eq:1.8}). We define the $\Phi(A)$(see (\ref{eq:3.19})) to describe the mixing effect, and establish the estimate of $\Phi(A)$, the main ideas is based on the RAGE theorem, energy estimate and contradiction, the details can be seen in Lemma \ref{lem:3.5}.

\vskip .05in

In the following, we briefly state our main ideas of the proof. Firstly, we establish the global $L^\infty$ estimate of solution to Equation (\ref{eq:1.8}) by global extension, the main idea is from abstract bootstrap principle (see \cite{Tao.2006}). The local $L^2$ and $ L^\infty$ estimate of the solution are obtained by energy method. If $A$ is large enough, then there exist a maximum time $T^\ast$, such that for any $t\in [0,T^\ast]$, the same local $L^2$ and $ L^\infty$ estimate of the solution are still true. it means that the time set of satisfies the local $L^2$ and $ L^\infty$ estimate is nonempty and close. And based on  the mixing effect when $A$ is large enough and damping effect of degenerate dissipation, at the time $T^\ast$, we show that the local $L^2$ and $ L^\infty$ estimate of solution to Equation (\ref{eq:1.8}) are controlled by initial data, so the time set of the local $L^2$ and $ L^\infty$ estimate can be extended. Then the time set is also a open set. Combining to the time set is nonempty, close and open, then $T^\ast=\infty$ by bootstrap argument. Next, we obtain the $W^{3,\infty}$ estimate of solution to Equation (\ref{eq:1.8}) by standard energy method. Based on the continuation argument, the global classical solution is obvious.

\vskip .05in

This paper is organized as follows. In Section 2, we introduce functional space and mixing effect. Some notations and useful properties are also introduced in this section. In Section 3, we obtain the global $L^\infty$ estimate of solution to Equation (\ref{eq:1.8}) by bootstrap argument. In Section 4, we establish the $W^{3,\infty}$ estimate of solution to Equation (\ref{eq:1.8}). In Appendix, we give the proof of Lemma \ref{lem:2.3} by some semigroup estimate method.

\vskip .05in

Throughout the paper, $C$ stands for universal constant that may change from line to line.

\vskip .2in

\section{Preliminaries}
In what follows, we provide some the auxiliary results and notations.

\subsection{Functional spaces and notations} We write $L^p(\mathbb{T}^d)$ for the usual  Lebesgue space, the norm for the $L^p$ space is denoted as $\|\cdot\|_{L^p}$, it means
$$
\|f\|_{L^p}=\left( \int_{\mathbb{T}^d}|f|^pdx\right)^{\frac{1}{p}},
$$
with natural adjustment when $p=\infty$. The homogeneous Sobolev norm  $\|\cdot\|_{\dot{W}^{s,p}}$ is defined by
$$
\|f\|_{\dot{W}^{s,p}}=\|D^{s}f\|_{L^p},
$$
where $s\geq 0, p\geq 1$ and $ D$ is any partial  derivative. The non-homogeneous Sobolev norm  $\|\cdot\|_{W^{s,p}}$ is
$$
\|f\|_{W^{s,p}}=\|f\|_{\dot{W}^{s,p}}+\|f\|_{L^p}.
$$
In particular, if $p=2$, the homogeneous Sobolev norm  $\|\cdot\|_{\dot{W}^{s,2}}=\|\cdot\|_{\dot{H}^{s}}$, and
$$
\|f\|^2_{\dot{H}^{s}(\mathbb{T}^d)}=\|(-\Delta)^{\frac{s}{2}}f\|^2_{L^2(\mathbb{T}^d)}=\sum_{k\in \mathbb{Z}^d\backslash \{0\}}|k|^{2s}|\hat{f}(k)|^2.
$$
The non-homogeneous Sobolev norm  $\|\cdot\|_{W^{s,2}}=\|\cdot\|_{H^{s}}$, and
$$
\|f\|_{H^{s}}=\|f\|_{L^2}+\|f\|_{\dot{H}^{s}}.
$$
The negative-index Sobolev space norm $\|\cdot\|_{H^{-s}}$ is defined as
$$
\|f\|^2_{\dot{H}^{-s}(\mathbb{T}^d)}=\sum_{k\in \mathbb{Z}^d\backslash \{0\}}|k|^{-2s}|\hat{f}(k)|^2,\ \ \ s>0.
$$

In this paper,  some basic energy inequalities are used in the proof, we can refer to \cite{Evans.2010,Friedman.1969} for more details.

\subsection{Mixing  effect}
Given an incompressible vector field $u=u(x)$, which is Lipschitz in spatial variables. If we defined the trajectories map by (see \cite{Constantin.2008, Kiselev.2016,Shi.2019})
$$
\frac{d}{dt} \Phi_t(x)=u(\Phi_t(x)), \quad \Phi_0(x)=x.
$$
Then define a unitary operator $U=U^t$ acting on $L^2(\mathbb{T}^d)$ as follows
$$
U^t f(x)=f(\Phi_t^{-1}(x)), \ \  f(x)\in L^2(\mathbb{T}^d).
$$
We take
\begin{equation}\label{eq:2.1}
X=\left\{ f\in L^2(\mathbb{T}^d)\big| \int_{\mathbb{T}^d}f(x)dx=0\right\},
\end{equation}
and for a fixed $t>0$, denote
$$
U=U^t.
$$

First, we introduce the definition of  weakly mixing (see \cite{Constantin.2008,Kiselev.2016}).

\begin{define}\label{def:2.1}
The incompressible flow $u$ is called weakly mixing, if $u=u(x)$ is smooth and the spectrum of the operator $U$ is purely continuous on $X$.
\end{define}

We denote by $0\leq\lambda_1\leq\lambda_2\leq\cdots\leq \lambda_n\leq\cdots$ the eigenvalues of $-\Delta$, without loss generality, suppose that
\begin{equation}\label{eq:2.2}
\lambda_n=n^2, \ \ \ n=1,2\cdots.
\end{equation}
Let us also denote by $P_N$ the orthogonal projection operator on the subspace formed by Fourier modes $|k|\leq N$:
\begin{equation}\label{eq:2.3}
P_Nf(x)=\sum_{|k|\leq N}\hat{f}(k)e^{ikx}.
\end{equation}
If we defined
\begin{equation}\label{eq:2.4}
S=\{\phi \in X\big| \|\phi\|_{L^2}=1\}.
\end{equation}

The following lemma is an extension of the well-know RAGE theorem (see \cite{Constantin.2008,Cycon.1987, Kiselev.2016}).

\begin{lemma}\label{lem:2.2}
Let $U^t$ be a unitary operator with purely continuous spectrum defined on $L^2(\mathbb{T}^d)$. Let for any $\phi\in S$, Then for every $N$ and $\sigma>0$, there exists $T_c=T(N, \sigma, \phi, U^t)$ such that for all $T\geq T_c$, one has
$$
\frac{1}{T}\int_0^T \|P_N U^t \phi\|_{L^2}^2dt\leq\sigma,
$$
where $P_N$ is defined in (\ref{eq:2.3}) and $S$ is defined in (\ref{eq:2.4}).
\end{lemma}

Next, we give a useful lemma for the estimation of operator $e^{-tAu\cdot \nabla}$, it is as follows
\begin{lemma}\label{lem:2.3}
Let $u$ is weakly mixing. The $e^{-tAu\cdot \nabla}$ is a semigroup operator with operator $-Au\cdot \nabla$. There exist a $A_1>0$, if $A>A_1$, then for any $t\geq0$ and $f\in \mathcal{S}(\mathbb{T}^d)$, we have
$$
\|P_Ne^{-tAu\cdot \nabla}f\|_{L^2}\leq Ce^{N^2 t}\|P_N f\|_{L^2},
$$
where $P_N$ is defined in (\ref{eq:2.3}) and $C>1$ is a fixed constant.
\end{lemma}

\begin{remark}\label{rem:3}
In this paper, the RAGE theorem is different from that in reference \cite{Constantin.2008, Shi.2019,Kiselev.2016},  because we can not find a compact set to fix the time $T_c=T(N, \sigma, \mathcal{K}, U^t)$. In fact, the time $T_c=T(N, \sigma, \phi, U^t)$(see Lemma \ref{lem:2.2}) is dependent on $\|\phi\|_{L^2}$. In the proof of RAGE theorem, under the assumption of $\|\phi\|_{L^2}=1$, the estimate of Lemma \ref{lem:2.2} become the estimate of operator norm, which is independent of $\phi$,  the details can refer to \cite{Cycon.1987}.
\end{remark}

\begin{remark}
The proof of Lemma \ref{lem:2.3}  needs to supplement some semigroups and resolvent estimates, where we omitte the proof. The details can referred to Appendix.
\end{remark}

\vskip .2in

\section{Global $L^\infty$ estimate}
In this section, we establish the global $L^\infty$ estimate of solution to Equation (\ref{eq:1.8}) by bootstrap argument. For initial data $\rho_0(x)$, we suppose
\begin{equation}\label{eq:3.1}
\rho_0(x)\geq 0,\ \ \ \rho_0(x)\in W^{3,\infty}(\mathbb{T}^d),
\end{equation}
and
\begin{equation}\label{eq:3.2}
\|\rho_0\|_{L^\infty}\leq C_{\infty},\ \ \ \ \|\rho_0\|_{L^2}\leq B_0.
\end{equation}
Then one has
$$
\rho(t,x)\geq 0,
$$
and compare with Equation (\ref{eq:1.1}) in the case of $\alpha>0$, the $L^1$ norm of solution to Equation (\ref{eq:1.8}) is not conservation. Define
\begin{equation}\label{eq:1.9}
\overline{\rho}(t)=\frac{1}{|\mathbb{T}^d|}\int_{\mathbb{T}^d}\rho(t,x)dx,\ \ \ \ \ \ \overline{\rho}_0=\frac{1}{|\mathbb{T}^d|}\int_{\mathbb{T}^d}\rho_0(x)dx,
\end{equation}
then we deduce by Equation (\ref{eq:1.8}) that
\begin{equation}\label{eq:1.10}
\frac{d}{dt}\overline{\rho}+\overline{\rho}=0,\ \ \ \ \ \ \overline{\rho}(t)=e^{-t}\overline{\rho}_0.
\end{equation}

\subsection{Local estimate}
First, we establish the local estimate of solution to Equation (\ref{eq:1.8}).

\begin{lemma}\label{lem:3.1}
Let $2\leq\beta<d, d>2$, $\rho(t,x)$ is the solution of Equation (\ref{eq:1.8}) with initial data $\rho_0(x)$. Suppose that the $\rho_0(x)$ satisfies (\ref{eq:3.1}) and (\ref{eq:3.2}). Then there exist a time $\tau_0>0$, such that for any $0\leq t \leq \tau_0$, one has
$$
\|\rho(t,\cdot)\|_{L^\infty}\leq 2C_{\infty}.
$$
\end{lemma}

\begin{proof}
Denote
$$
\widetilde{\rho}(t)=\rho(t,\overline{x}_t)=\max_{x\in \mathbb{T}^d}\rho(t,x).
$$
For any fixed $t\geq0$, using the vanishing of a derivation at the point of maximum, we see that
$$
\partial_t \rho(t,\overline{x}_t)=\frac{d}{dt}\widetilde{\rho}(t),\quad
(u\cdot \nabla\rho)(t,\overline{x}_t)=u\cdot \nabla\widetilde{\rho}(t)=0,
$$
and
$$
\left(\nabla\cdot(\rho B(\rho))\right)(t,\overline{x}_t)
=\widetilde{\rho}\Delta K\ast\rho(t,\overline{x}_t).
$$
Thus, we deduce by the Equation (\ref{eq:1.8}) that  the evolution of $\widetilde{\rho}$ follows
\begin{equation}\label{eq:3.3}
\frac{d}{dt}\widetilde{\rho}+\widetilde{\rho}+\widetilde{\rho}\Delta K\ast\rho(t,\overline{x}_t)=0.
\end{equation}
Combining H\"{o}lder's inequality, Young's inequality and $\beta\in [2,d)$, to obtain
\begin{equation}\label{eq:3.4}
\big|\widetilde{\rho}\Delta K\ast\rho(t,\overline{x}_t)\big|\leq C\|\Delta K\|_{L^1}\|\rho\|^2_{L^\infty}\leq C_0\widetilde{\rho}^2,
\end{equation}
where
$$
C_0=C\|\Delta K\|_{L^1}.
$$
Then we deduce by (\ref{eq:3.3}) and (\ref{eq:3.4}) that
\begin{equation}\label{eq:3.5}
\frac{d}{dt}\widetilde{\rho}\leq -\widetilde{\rho}+C_0\widetilde{\rho}^2\leq C_0\widetilde{\rho}^2.
\end{equation}
As $\|\rho_0\|_{L^\infty}\leq C_\infty$, then we define
\begin{equation}\label{eq:3.6}
\tau_0=\min\left\{ \frac{1}{2C_0C_\infty},T_0 \right\},
\end{equation}
where $T_0$ is  time of local existence, see Remark \ref{rem:1}. Then by solving the differential inequality in (\ref{eq:3.5}), for any $0\leq t\leq \tau_0$, one has
$$
\|\rho(t,\cdot)\|_{L^\infty}\leq 2C_{\infty}.
$$
This completes the proof of Lemma \ref{lem:3.1}.
\end{proof}

\begin{remark}\label{rem:1}
The local existence of Equation (\ref{eq:1.8}) can be established by standard method in the case of $2\leq \beta<d, d>2$. And local time $T_0$ is independent of the constant $A$, the details discussion  can be referred to \cite{Ascasibar.2013,Li.2010}.
\end{remark}

Next, we establish the local $L^2$ estimate of solution to Equation (\ref{eq:1.8}).

\begin{lemma}\label{lem:3.2}
Let $2\leq\beta<d, d>2$, $\rho(t,x)$ is the solution of Equation (\ref{eq:1.8}) with initial data $\rho_0(x)$. Suppose that the $\rho_0(x)$ satisfies (\ref{eq:3.1}) and (\ref{eq:3.2}). Then there exist a time $\tau_1>0$, for any $0\leq t \leq \tau_1$, one has
$$
\|\rho(t,\cdot)-\overline{\rho}(t)\|_{L^2}\leq 2(B_0^2-\overline{\rho}_0^2)^{\frac{1}{2}},
$$
where $\overline{\rho}(t)$ and $\overline{\rho}_0$ are defined in (\ref{eq:1.9}) and (\ref{eq:1.10}).
\end{lemma}

\begin{proof}
Let us multiply both sides of (\ref{eq:1.8}) by $\rho-\overline{\rho}$ and integrate over $\mathbb{T}^d$, and combining to the impressibility of $u$ and (\ref{eq:1.10}),  we obtain that
\begin{equation}\label{eq:3.7}
\frac{1}{2}\frac{d}{dt}\|\rho-\overline{\rho}\|_{L^2}^2+\|\rho-\overline{\rho}\|^2_{L^2}+\int_{\mathbb{T}^d}\nabla\cdot(\rho B(\rho))(\rho-\overline{\rho})dx=0.
\end{equation}
Combining (\ref{eq:1.2}), $\beta\in [2,d)$,  H\"{o}lder's inequality and Young's inequality, the third term of the left-hand side of (\ref{eq:3.7}) can be estimated as
\begin{equation}\label{eq:3.8}
\begin{aligned}
&\left|\int_{\mathbb{T}^d}\nabla\cdot(\rho B(\rho))(\rho-\overline{\rho})dx\right|\\
=&\left|\frac{1}{2}\int_{\mathbb{T}^d}\Delta K\ast \rho(\rho-\overline{\rho})^2dx+\overline{\rho}\int_{\mathbb{T}^d}\Delta K\ast (\rho-\overline{\rho})(\rho-\overline{\rho})dx\right|\\
\leq& C\|\Delta K\|_{L^1}\| \rho\|_{L^\infty}\|\rho-\overline{\rho}\|^2_{L^2}+C\overline{\rho}\|\Delta K\|_{L^1}\|\rho-\overline{\rho}\|^2_{L^2}\\
\leq&C(\|\rho\|_{L^\infty}+\overline{\rho})\|\rho-\overline{\rho}\|^2_{L^2}.
\end{aligned}
\end{equation}
According to (\ref{eq:1.10}), (\ref{eq:3.8}) and Lemma \ref{lem:3.1}, we imply that for any $t\in [0,\tau_0]$, one has
\begin{equation}\label{eq:3.9}
\left|\int_{\mathbb{T}^d}\nabla\cdot(\rho B(\rho))(\rho-\overline{\rho})dx\right|\leq C( C_\infty+\overline{\rho}_0)\|\rho-\overline{\rho}\|^2_{L^2}.
\end{equation}
Combining (\ref{eq:3.7}) and (\ref{eq:3.9}), for any $t\in [0,\tau_0]$, one get
\begin{equation}\label{eq:3.10}
\begin{aligned}
\frac{d}{dt}\|\rho-\overline{\rho}\|_{L^2}^2\leq -2\|\rho-\overline{\rho}\|^2_{L^2}+C(C_\infty+\overline{\rho}_0)\|\rho-\overline{\rho}\|^2_{L^2}.
\end{aligned}
\end{equation}
We denote
\begin{equation}\label{eq:3.11}
\tau_1=\min\left\{\tau_0, \frac{2\ln 2}{C(C_\infty+\overline{\rho}_0)}\right\}.
\end{equation}
By solving the differential inequality in (\ref{eq:3.10}), for any $0\leq t\leq\tau_1$, one get
$$
\|\rho(t,\cdot)-\overline{\rho}(t)\|_{L^2}\leq 2(B_0^2-\overline{\rho}_0^2)^{\frac{1}{2}}.
$$
This completes the proof of Lemma \ref{lem:3.2}.
\end{proof}

\begin{remark}\label{rem:4}
According to the Lemma  \ref{lem:3.1} and Lemma \ref{lem:3.2}, we deduce that for any $0\leq t\leq\tau_1$ one has
$$
\|\rho(t,\cdot)\|_{L^\infty}\leq 2C_{\infty},\ \ \ \ \ \|\rho(t,\cdot)-\overline{\rho}(t)\|_{L^2}\leq 2(B_0^2-\overline{\rho}_0^2)^{\frac{1}{2}}.
$$
\end{remark}

\subsection{Bootstrap argument}
In this section, we extend the local estimate to global estimate by bootstrap argument. We list the bootstrap assumptions as below:

\begin{assumption}\label{assm:3.3}
Let $2\leq\beta<d, d>2$, $\rho(t,x)$ is the solution of Equation (\ref{eq:1.8}) with initial data $\rho_0(x)$. Suppose that the $\rho_0(x)$ satisfies (\ref{eq:3.1}) and (\ref{eq:3.2}). There exist a $A_0$ is large enough, if $A>A_0$, let us define $T^\ast>0$ to be the end-point of the largest interval $t\in [0,T^\ast]$ such that the following hypotheses hold for all $0\leq t\leq T^\ast$
\begin{equation}\label{eq:3.12}
\|\rho(t,\cdot)\|_{L^\infty}\leq 2C_{\infty},
\end{equation}
\begin{equation}\label{eq:3.13}
\|\rho(t,\cdot)-\overline{\rho}(t)\|_{L^2}\leq 2(B_0^2-\overline{\rho}_0^2)^{\frac{1}{2}}.
\end{equation}
\end{assumption}

\begin{remark}\label{rem:5}
The constant $A_0\geq A_1$, and $A_0$  is determined by the proof, the details can be seen in the later. Here $A_1$  is defined in the Lemma \ref{lem:2.3}.
\end{remark}

\begin{remark}\label{rem:6}
According to the local $L^\infty$ estimate and $L^2$ estimate  of solution, we deduce that $T^\ast\geq \tau_1$. In fact, $T^\ast=\infty$ for large $A$. We prove it by contradiction. Thus, our analysis is based on $T^\ast<\infty$ in next section.
\end{remark}

First, we give an  approximation lemma.

\begin{lemma}\label{lem:3.4}
Let $2\leq\beta<d, d>2$, $u(x)$ is weakly mixing. The $\rho(t,x), \eta(t,x)$ are the solution of Equations (\ref{eq:1.8}) and (\ref{eq:1.11}) respectively, and the $\rho(t,x)$ satisfies the Assumption \ref{assm:3.3}. Then for any $0\leq s\leq s+t\leq T^\ast$, if $\rho(s,x)=\eta(s,x)$, there exist a finite positive constant $C>1$, such that
$$
\|P_{N}(\rho-\eta)(t+s,\cdot)\|_{L^2}\leq  CNe^{N^2 t}C_\infty(B^2_0-\overline{\rho}_0^2)^{\frac{1}{2}}t,
$$
where  $P_N$ is defined in (\ref{eq:2.3}).
\end{lemma}

\begin{proof}
Consider the Equations (\ref{eq:1.8}) and (\ref{eq:1.11}), to obtain
\begin{equation}\label{eq:3.14}
\partial_t(\rho-\eta)+Au\cdot \nabla (\rho-\eta)+(\rho-\eta)+\nabla\cdot(\rho B(\rho))=0.
\end{equation}
By Duhamel's principe, then the solution $\rho-\eta$  of (\ref{eq:3.14}) can be expressed as
\begin{equation}\label{eq:3.15}
(\rho-\eta)(t+s,x)=\int_{0}^{t}e^{-(Au\cdot \nabla+1) \tau}\left(\nabla\cdot(\rho B(\rho))(t+s-\tau,x)\right)d\tau.
\end{equation}
The operator $P_N$ is applied to (\ref{eq:3.15}), then one get
\begin{equation}\label{eq:3.16}
P_{N}(\rho-\eta)(t+s,x)=\int_{0}^{t}P_{N}e^{-(Au\cdot \nabla+1) \tau}\left(\nabla\cdot(\rho B(\rho))(t+s-\tau,x)\right)d\tau.
\end{equation}
Combining Lemma \ref{lem:2.3}, Assumption \ref{assm:3.3}, H\"{o}lder's inequality and Young's inequality, to obtain
\begin{equation}\label{eq:3.17}
\begin{aligned}
&\|P_{N}e^{-(Au\cdot \nabla+1) \tau}\left(\nabla\cdot(\rho B(\rho))(t+s-\tau,\cdot)\right)\|_{L^2}\\
\leq& \|P_{N}e^{-Au\cdot \nabla \tau}\left(\nabla\cdot(\rho B(\rho))(t+s-\tau,\cdot)\right)\|_{L^2}\\
\leq& Ce^{N^2 \tau}\|P_{N}\left(\nabla\cdot(\rho B(\rho))(t+s-\tau,\cdot)\right)\|_{L^2}\\
\leq& CNe^{N^2 \tau}\|\rho\nabla K\ast(\rho-\overline{\rho})(t+s-\tau,\cdot)\|_{L^2}\\
\leq& CNe^{N^2 \tau}\|\rho(t+s-\tau,\cdot)\|_{L^\infty}\|\nabla K\|_{L^1}\|(\rho-\overline{\rho})(t+s-\tau,\cdot)\|_{L^2}.
\end{aligned}
\end{equation}
For any $0\leq \tau\leq t$, then $t+s-\tau\in [0,T^\ast]$. We deduce by Assumption \ref{assm:3.3} that
\begin{equation}\label{eq:3.18}
\|\rho(t+s-\tau,\cdot)\|_{L^\infty}\leq 2C_\infty,\ \ \ \|(\rho-\overline{\rho})(t+s-\tau,\cdot)\|_{L^2}\leq 2(B^2_0-\overline{\rho}_0^2)^{\frac{1}{2}}.
\end{equation}
As $\|\nabla K\|_{L^1}$ is bounded, then for any $0\leq s\leq s+t\leq T^\ast$, combining (\ref{eq:3.16})-(\ref{eq:3.18}), we have
$$
\begin{aligned}
&\|P_{N}(\rho-\eta)(t+s,\cdot)\|_{L^2}\\
\leq& \int_{0}^{t}\|P_{N}e^{-(Au\cdot \nabla+1) \tau}\left(\nabla\cdot(\rho B(\rho))(t+s-\tau,\cdot)\right)\|_{L^2}d\tau\\
\leq& \int_{0}^{t}CNe^{N^2 \tau}\|\rho(t+s-\tau,\cdot)\|_{L^\infty}\|\nabla K\|_{L^1}\|(\rho-\overline{\rho})(t+s-\tau,\cdot)\|_{L^2}d\tau\\
\leq& 4C\|\nabla K\|_{L^1}Ne^{N^2 t}C_\infty(B^2_0-\overline{\rho}_0^2)^{\frac{1}{2}}t\\
\leq& CNe^{N^2 t}C_\infty(B^2_0-\overline{\rho}_0^2)^{\frac{1}{2}}t.
\end{aligned}
$$
This completes the proof of Lemma \ref{lem:3.4}.
\end{proof}

Next, we improve the $L^\infty$ estimate and $L^2$ estimate  of solution to Equation (\ref{eq:1.8}) by mixing effect. Let us denote $\rho(t,x)$ is the solution of Equation (\ref{eq:1.8}) with initial data $\rho_0(x)$. We define
\begin{equation}\label{eq:3.19}
\Phi(A)=\frac{\|\Lambda^{\beta-d}(\rho-\overline{\rho})\|^2_{L^2}}{\|\rho-\overline{\rho}\|^2_{L^2}}
\end{equation}
to analysis the mixing effect in Equation (\ref{eq:1.8}). Here $\Phi(A)$ can be used to describe the frequency distribution of the solution of the Equation (\ref{eq:1.8}). Based on the idea of contradiction and RAGE theorem, we establish the estimate of $\Phi(A)$. The main lemma is as follows

\begin{lemma}\label{lem:3.5}
Let $2\leq\beta<d, d> 2$, $\rho(t,x)$ is the solution of Equation (\ref{eq:1.8}) with initial data $\rho_0(x)$ and $u(x)$ is weakly mixing.  For any fixed constant $N$  is large and $\epsilon_0$ is small, if the $\rho(t,x)$ satisfies the Assumption \ref{assm:3.3}, then there exist set $\Sigma(t)\subset [0,T^\ast]$, such that for all $t\in\Sigma(t)$, one has
$$
\Phi(A)\geq 2\lambda_N^{\beta-d}.
$$
Then
\begin{equation}\label{eq:3.20}
|\Sigma(t)|<\epsilon_0T^\ast,
\end{equation}
where $\Phi(A)$ is defined in (\ref{eq:3.19}) and $|\Sigma(t)|$ is measure of $\Sigma(t)$.
\end{lemma}

\begin{proof}
For fixed constant $N$  is large and $\epsilon_0$ is small, define
\begin{equation}\label{eq:3.21}
B_1=C\lambda^{\beta-d}_N\|\rho_0-\overline{\rho}_0\|_{L^2}, \ \ \ k_0=\left[\frac{4}{\epsilon_0}\right]+1,
\end{equation}
where $C>0$ is a fix  constant, $N$ and $\epsilon_0$ are choose in the later. Without loss of generality, we assume that for any $t\in  [0,T^\ast]$, one has
\begin{equation}\label{eq:3.22}
\|\rho(t,\cdot)-\overline{\rho}\|_{L^2}\geq B_1,
\end{equation}
the details can be referred in Remark \ref{rem:9}. Next, We give the proof by contradiction. If (\ref{eq:3.20}) is not true, then there exist a sequence $\{A_n\}_{n=1}^{\infty}$ and a set $\Sigma_1(t)\subset [0,T^\ast]$, such that
$$
\lim_{n\rightarrow \infty}A_n=\infty, \ \ \ |\Sigma_1(t)|=\epsilon_0T^\ast.
$$
And for any $A_n$ and $t\in \Sigma_1(t)$, we have
\begin{equation}\label{eq:3.23}
\frac{\|\Lambda^{\beta-d}(\rho-\overline{\rho})\|^2_{L^2}}{\|\rho-\overline{\rho}\|^2_{L^2}}\geq 2\lambda_N^{\beta-d}.
\end{equation}
Firstly, we deduce by (\ref{eq:3.23}) that the following two conclusions

\vskip .05in

(1) We claim that for any $A_n$ and $t\in \Sigma_1(t)$, one has
\begin{equation}\label{eq:3.24}
\|P_N(\rho-\overline{\rho})\|^2_{L^2}\geq \lambda_N^{\beta-d}\|\rho(t,\cdot)-\overline{\rho}\|^2_{L^2}.
\end{equation}
If (\ref{eq:3.24}) is not true, then there exist $A_{n_0}$ and $t_0\in \Sigma_1(t)$, such that
\begin{equation}\label{eq:3.25}
\|P_N(\rho-\overline{\rho})(t_0)\|^2_{L^2}<\lambda_N^{\beta-d}\|(\rho-\overline{\rho})(t_0)\|^2_{L^2}.
\end{equation}
Then we deduce by (\ref{eq:3.25}) that
$$
\begin{aligned}
\|\Lambda^{\beta-d}(\rho-\overline{\rho})(t_0)\|^2_{L^2}& \leq \|\Lambda^{\beta-d}P_N(\rho-\overline{\rho})(t_0)\|^2_{L^2}+\|\Lambda^{\beta-d}(I-P_N)(\rho-\overline{\rho})(t_0)\|^2_{L^2}\\
&\leq \|P_N(\rho-\overline{\rho})(t_0)\|^2_{L^2}+\lambda^{\beta-d}_N\|(I-P_N)(\rho-\overline{\rho})(t_0)\|^2_{L^2}\\
&<\lambda_N^{\beta-d}\|(\rho-\overline{\rho})(t_0)\|^2_{L^2}+\lambda_N^{\beta-d}\|(I-P_N)(\rho-\overline{\rho})(t_0)\|^2_{L^2}\\
&<2\lambda_N^{\beta-d}\|(\rho-\overline{\rho})(t_0)\|^2_{L^2}.
\end{aligned}
$$
This is in contradiction with (\ref{eq:3.23}), we finish the proof of (\ref{eq:3.24}).

\vskip .05in

(2) We claim that any $A_n$, there exist $t_0\in \Sigma_1(t)$ and $\tau\leq \frac{1}{M}<\frac{\epsilon_0}{2} T^\ast$, where $M$ is arbitrary and large enough, such that
\begin{equation}\label{eq:3.26}
\frac{1}{\tau}\int_{t_0}^{t_0+\tau}\|P_N(\rho(t,\cdot)-\overline{\rho})\|^2_{L^2}dt\geq \frac{2\lambda^{\beta-d}_N}{k_0}B^2_1,
\end{equation}
where $[t_0,t_0+\tau)\subset [0,T^\ast]$. If (\ref{eq:3.26}) is not true, then there exist $A_{n_1}$ and for all $\tau\leq \frac{1}{M}<\frac{\epsilon_0}{2} T^\ast$ and $[t_i, t_i+\tau)\subset [0,T^\ast], t_i\in \Sigma_1(t)$, and
$$
[t_i, t_i+\tau)\cap [t_j, t_j+\tau)=\emptyset,\ \ \ i=0,1,2,\cdots,
$$
such that
\begin{equation}\label{eq:3.27}
\frac{1}{\tau}\int_{t_i}^{t_i+\tau}\|P_N(\rho(t,\cdot)-\overline{\rho})\|^2_{L^2}dt<\frac{2\lambda^{\beta-d}_N}{k_0}B^2_1.
\end{equation}
If we denote
$$
E=\cup_{i}[t_i, t_i+\tau), \ \ \ E_1=E\cap \Sigma_1(t)=\bigcup_{i}([t_i, t_i+\tau)\cap \Sigma_1(t)),
$$
then we deduce that the definition of $E_1$ by
\begin{equation}\label{eq:3.28}
\frac{\epsilon_0}{2}T^\ast< |E_1|\leq \epsilon_0T^\ast.
\end{equation}
Then we have
\begin{equation}\label{eq:3.29}
\begin{aligned}
\int_{E_1}\|P_N(\rho(t,\cdot)-\overline{\rho})\|^2_{L^2}dt&\leq \int_{E}\|P_N(\rho(t,\cdot)-\overline{\rho})\|^2_{L^2}dt\\
&=\sum_{i}\int_{t_i}^{t_i+\tau}\|P_N(\rho(t,\cdot)-\overline{\rho})\|^2_{L^2}dt\\
&<\frac{2\lambda^{\beta-d}_N}{k_0}B^2_1\sum_{i}\tau\leq \frac{2\lambda^{\beta-d}_N}{k_0}B^2_1T^\ast.
\end{aligned}
\end{equation}
Combining (\ref{eq:3.22}) and (\ref{eq:3.24}), we imply that for any $t\in \Sigma_1(t)$, one has
\begin{equation}\label{eq:3.30}
\|P_N(\rho-\overline{\rho})\|^2_{L^2}\geq \lambda_N^{\beta-d}B_1^2.
\end{equation}
Since $E_1\subset \Sigma_1(t)$, then we deduce by (\ref{eq:3.21}), (\ref{eq:3.30}) and (\ref{eq:3.28}) that
\begin{equation}\label{eq:3.31}
\int_{E_1}\|P_N(\rho(t,\cdot)-\overline{\rho})\|^2_{L^2}dt\geq \lambda^{\beta-d}_NB^2_1|E_1|\geq \frac{\lambda^{\beta-d}_N\epsilon_0}{2}B^2_1T^\ast\geq \frac{2\lambda^{\beta-d}_N}{k_0}B^2_1T^\ast.
\end{equation}
Obviously, (\ref{eq:3.29}) and (\ref{eq:3.31}) are contradictory, so we finish the proof of (\ref{eq:3.26}).

\vskip .05in

Next, we prove that the  (\ref{eq:3.26}) is not true by RAGE theorem. we consider the equation
\begin{equation}\label{eq:3.32}
\partial_t\rho+A_n u\cdot \nabla \rho+\rho+\nabla\cdot(\rho B(\rho))=0,\ \ \ \rho_0(x)=\rho(t_0,x),
\end{equation}
and
\begin{equation}\label{eq:3.33}
\partial_t\eta+A_n u\cdot \nabla \eta+\eta=0, \ \ \ \eta_0(x)=\rho(t_0,x).
\end{equation}
Then for the solution $\rho(t,x)$ of equation (\ref{eq:3.32}), we deduce by Assumption \ref{assm:3.3} that
$$
\|\rho(t,\cdot)-\overline{\rho}(t)\|_{L^2}\leq 2(B_0^2-\overline{\rho}^2)^{\frac{1}{2}},\quad \|\rho(t,\cdot)\|_{L^\infty}\leq 2C_{\infty},\quad 0\leq t \leq T^\ast.
$$
Next, we consider the equation (\ref{eq:3.33}), according to the definition of $U^{t}$, one has
$$
\eta(t_0+t,x)-\overline{\rho}=e^{-t}U^{A_nt}(\rho(t_0,x)-\overline{\rho}(t_0)).
$$
Here $U^t$ is the unitary operator associated with weakly mixing flow $u$ and fix $\sigma=\frac{\lambda^{\beta-d}_N B_1^2}{8k_0(B^2_0-\overline{\rho}^2_0)}$, we get $T_c=T_c(N, \sigma, \rho(t_0), U)$, which is the time provided by Lemma \ref{lem:2.2}. There exist $A_0$ such that $\tau=\frac{T_c}{A_0}$, then for $A_n\geq A_0$, by RAGE theorem, one has
\begin{equation}\label{eq:3.34}
\begin{aligned}
&\frac{1}{\tau}\int_{0}^{\tau}\|P_N(\eta(t_0+t,x)-\overline{\rho}(t_0+t))\|_{L^2}^2dt\\
=&\frac{1}{\tau}\int_{0}^{\tau}\|P_Ne^{-t} U^{A_nt}(\rho(t_0,x)-\overline{\rho}(t_0))\|_{L^2}^2dt \\
=&\frac{\|\rho(t_0,x)-\overline{\rho}\|_{L^2}^2}{\tau}\int_{0}^{\tau}\|P_Ne^{-t} U^{A_nt}\frac{(\rho(t_0,x)-\overline{\rho})}{\|\rho(t_0,x)-\overline{\rho}\|_{L^2}}\|_{L^2}^2dt\\
=&\frac{\|\rho(t_0,x)-\overline{\rho}\|_{L^2}^2}{A_n\tau}\int_{0}^{\tau}\|P_N U^{A_nt}\frac{(\rho(t_0,x)-\overline{\rho}}{\|\rho(t_0,x)-\overline{\rho}\|_{L^2}}\|_{L^2}^2dA_nt\\
=&\frac{\|\rho(t_0,x)-\overline{\rho}\|_{L^2}^2}{A_n\tau}\int_{0}^{A_n\tau}\|P_N U^{s}\frac{(\rho(t_0,x)-\overline{\rho})}{\|\rho(t_0,x)-\overline{\rho}\|_{L^2}}\|_{L^2}^2ds\\
\leq& \sigma \|\rho(t_0,x)-\overline{\rho}\|_{L^2}^2\leq \frac{\lambda^{\beta-d}_N}{4k_0}B_1^2,
\end{aligned}
\end{equation}
where $A_n\tau\geq T_c$. For any $t\in [0,\tau]$,  and $M$ is large enough, one has
$$
CN^2e^{2N^2 t}C_\infty(B^2_0-\overline{\rho}_0^2)t^2\leq CN^2e^{2N^2 \tau}C^2_\infty(B^2_0-\overline{\rho}_0^2)\tau^2\leq \frac{\lambda^{\beta-d}_N}{4k_0}B_1^2.
$$
According to Lemma \ref{lem:3.4}, one has
\begin{equation}\label{eq:3.35}
\|P_N(\rho-\eta)(t_0+t,\cdot)\|^2_{L^2}\leq \frac{\lambda^{\beta-d}_N}{4k_0}B^2_1.
\end{equation}
Furthermore, we deduce by (\ref{eq:3.34}) and (\ref{eq:3.35}) that
\begin{equation}\label{eq:3.36}
\begin{aligned}
\frac{1}{\tau}\int_{0}^{\tau}\|P_N(\rho(t_0+t,\cdot)-\overline{\rho})\|_{L^2}^2dt
& \leq\frac{2}{\tau}\int_{0}^{\tau}\|P_N(\eta(t_0+t,\cdot)-\overline{\rho})\|_{L^2}^2dt\\
&\quad +\frac{2}{\tau}\int_{0}^{\tau}\|P_N(\rho(t_0+t,\cdot)-\eta(t_0+t,\cdot))\|_{L^2}^2dt\\
&\leq\frac{\lambda^{\beta-d}_N}{k_0}B^2_1.
\end{aligned}
\end{equation}
Then we deduce by (\ref{eq:3.36}) that
\begin{equation}\label{eq:3.37}
\frac{1}{\tau}\int_{t_0}^{t_0+\tau}\|P_N(\rho(t,\cdot)-\overline{\rho})\|^2_{L^2}dt=\frac{1}{\tau}\int_{0}^{\tau}\|P_N(\rho(t_0+t,\cdot)-\overline{\rho})\|_{L^2}^2dt\leq \frac{\lambda^{\beta-d}_N}{k_0}B^2_1.
\end{equation}
Obviously, (\ref{eq:3.26}) and (\ref{eq:3.37}) are contradictory. Then the (\ref{eq:3.26}) is not true. This completes the proof of Lemma \ref{lem:3.5}.
\end{proof}

\begin{remark}\label{rem:7}
In fact, for any $t\in [0,T^\ast]\backslash \Sigma(t)$, we can prove that
$$
\lim_{A\rightarrow \infty}\Phi(A)=0.
$$
And according to the definition of $\Phi(A)$ in (\ref{eq:3.19}), we know that the $\Phi(A)$ depends on $\beta, d, \rho_0, A$, then we deduce that $A_0=A(\beta,d, \rho_0)$.
\end{remark}

\begin{remark}\label{rem:8}
According the definition of $\Phi(A)$ and $2\leq \beta<d$, we know that
$$
\Phi(A)\leq 1.
$$
\end{remark}

we will show the following proposition to improve the Assumption \ref{assm:3.3} by mixing effect.

\begin{proposition}\label{prop:3.6}
Let $2\leq\beta<d, d> 2$, $\rho(t,x)$ is the solution of Equation (\ref{eq:1.8}) with initial data $\rho_0(x)$. If $u(x)$ is weakly mixing and the $\rho(t,x)$ satisfies the Assumption \ref{assm:3.3}.  Then for $A$ is large enough, one has
$$
\|\rho(T^\ast,\cdot)\|_{L^\infty}\leq C_{\infty},
$$
$$
\|\rho(T^\ast,\cdot)-\overline{\rho}(T^\ast)\|_{L^2}\leq (B_0^2-\overline{\rho}_0^2)^{\frac{1}{2}}.
$$
\end{proposition}

\begin{proof}
(1) {\bf $L^2$ estimate}: For the fourth term of left-hand side of (\ref{eq:3.7}) and (\ref{eq:3.8}), one has
\begin{equation}\label{eq:3.38}
\begin{aligned}
&\int_{\mathbb{T}^d}\nabla\cdot(\rho B(\rho))(\rho-\overline{\rho})dx\\
=&\frac{1}{2}\int_{\mathbb{T}^d}\Delta K\ast \rho(\rho-\overline{\rho})^2dx+\overline{\rho}\int_{\mathbb{T}^d}\Delta K\ast (\rho-\overline{\rho})(\rho-\overline{\rho})dx\\
=&\frac{1}{2}\int_{\mathbb{T}^d}\Delta K\ast (\rho-\overline{\rho})(\rho-\overline{\rho})^2dx+\overline{\rho}\int_{\mathbb{T}^d}\Delta K\ast (\rho-\overline{\rho})(\rho-\overline{\rho})dx.
\end{aligned}
\end{equation}
By H\"{o}lder's inequality and Young's inequality, the first term of right-hand side of (\ref{eq:3.38}) can be estimated as
\begin{equation}\label{eq:3.39}
\begin{aligned}
&\left|\frac{1}{2}\int_{\mathbb{T}^d}\Delta K\ast (\rho-\overline{\rho})(\rho-\overline{\rho})^2dx\right|\\
=&\left|\frac{1}{2}\int_{\mathbb{T}^d}\Lambda^{2-\frac{\beta-d}{2}} K\ast \Lambda^{\frac{\beta-d}{2}}(\rho-\overline{\rho})(\rho-\overline{\rho})^2dx\right|\\
\leq& C\|\Lambda^{2-\frac{\beta-d}{2}} K\|_{L^1}\|\Lambda^{\frac{\beta-d}{2}}(\rho-\overline{\rho})\|_{L^2}\|\rho-\overline{\rho}\|_{L^2}\|\rho-\overline{\rho}\|_{L^\infty}\\
\leq& CC_\infty\|\Lambda^{\frac{\beta-d}{2}}(\rho-\overline{\rho})\|_{L^2}\|\rho-\overline{\rho}\|_{L^2},
\end{aligned}
\end{equation}
and
\begin{equation}\label{eq:3.40}
\begin{aligned}
&\left|\overline{\rho}\int_{\mathbb{T}^d}\Delta K\ast (\rho-\overline{\rho})(\rho-\overline{\rho})dx\right|\\
=&\left|\overline{\rho}\int_{\mathbb{T}^d}\Lambda^{2-\frac{\beta-d}{2}} K\ast \Lambda^{\frac{\beta-d}{2}}(\rho-\overline{\rho})(\rho-\overline{\rho})dx\right|\\
\leq& C\overline{\rho}_0\|\Lambda^{\frac{\beta-d}{2}}(\rho-\overline{\rho})\|_{L^2}\|\rho-\overline{\rho}\|_{L^2}.
\end{aligned}
\end{equation}
Combining (\ref{eq:3.7}) and (\ref{eq:3.38})-(\ref{eq:3.40}), to obtain $t\in [0,T^\ast]$
\begin{equation}\label{eq:3.41}
\frac{d}{dt}\|\rho-\overline{\rho}\|_{L^2}^2\leq -2\|\rho-\overline{\rho}\|^2_{L^2}+C(C_\infty+\overline{\rho}_0)\|\Lambda^{\frac{\beta-d}{2}}(\rho-\overline{\rho})\|_{L^2}\|\rho-\overline{\rho}\|_{L^2}.
\end{equation}
By Gagliardo-Nirenberg inequality, one has
\begin{equation}\label{eq:3.42}
\|\Lambda^{\frac{\beta-d}{2}}(\rho-\overline{\rho})\|_{L^2}\leq C\|\Lambda^{\beta-d}(\rho-\overline{\rho})\|^{\frac{1}{2}}_{L^2}\|\rho-\overline{\rho}\|^{\frac{1}{2}}_{L^2}.
\end{equation}
For the second term of right-hand side of (\ref{eq:3.41}), we deduce by (\ref{eq:3.19}) and (\ref{eq:3.42}) that
\begin{equation}\label{eq:3.43}
\begin{aligned}
&C(C_\infty+\overline{\rho}_0)\|\Lambda^{\frac{\beta-d}{2}}(\rho-\overline{\rho})\|_{L^2}\|\rho-\overline{\rho}\|_{L^2}\\
\leq& C(C_\infty+\overline{\rho}_0)\|\Lambda^{\beta-d}(\rho-\overline{\rho})\|^{\frac{1}{2}}_{L^2}\|\rho-\overline{\rho}\|^{\frac{3}{2}}_{L^2}\\
\leq& C(C_\infty+\overline{\rho}_0)(\Phi(A))^{\frac{1}{4}}\|\rho-\overline{\rho}\|^{2}_{L^2}.
\end{aligned}
\end{equation}
According to (\ref{eq:3.41}) and (\ref{eq:3.43}), one has
\begin{equation}\label{eq:3.44}
\frac{d}{dt}\|\rho-\overline{\rho}\|_{L^2}^2\leq -2\|\rho-\overline{\rho}\|^2_{L^2}+ C(C_\infty+\overline{\rho}_0)(\Phi(A))^{\frac{1}{4}}\|\rho-\overline{\rho}\|^{2}_{L^2}\\
\end{equation}
If $\rho(t,x)$ satisfies (\ref{eq:3.22}), then we deduce by Lemma \ref{lem:3.5}, Remark \ref{rem:8} and Gronwall's inequality, one has
\begin{equation}\label{eq:3.45}
\begin{aligned}
\|(\rho-\overline{\rho})(T^\ast,\cdot)\|_{L^2}^2&\leq\|(\rho_0-\overline{\rho}_0)\|_{L^2}^2\exp\left(\int_{0}^{T^\ast}-2+ C(C_\infty+\overline{\rho}_0)(\Phi(A))^{\frac{1}{4}}dt \right)\\
&\leq \|(\rho_0-\overline{\rho}_0)\|_{L^2}^2e^{-2T^\ast}\exp\left(\int_{0}^{T^\ast}C(C_\infty+\overline{\rho}_0)(\Phi(A))^{\frac{1}{4}}dt \right)\\
&\leq \|(\rho_0-\overline{\rho}_0)\|_{L^2}^2e^{-2T^\ast}\exp\left(\int_{\Sigma(t)}C_1(\Phi(A))^{\frac{1}{4}}dt+\int_{[0,T^\ast]\backslash \Sigma(t)}C_1(\Phi(A))^{\frac{1}{4}}dt\right)\\
&\leq \|(\rho_0-\overline{\rho}_0)\|_{L^2}^2e^{-2T^\ast}e^{C_1\epsilon_0T^\ast}e^{C_1T^\ast(\Phi(A))^{\frac{1}{4}}},
\end{aligned}
\end{equation}
where $C_1=C(C_\infty+\overline{\rho}_0)$, choose $\epsilon_0$  is small,  $N$ and $A$ are large enough, such that
\begin{equation}\label{eq:3.46}
C_1\epsilon_0+C_1(\Phi(A))^{\frac{1}{4}}<2.
\end{equation}
Combining (\ref{eq:3.45}) and (\ref{eq:3.46}), to obtain
$$
\|(\rho-\overline{\rho})(T^\ast)\|_{L^2}\leq  \|(\rho_0-\overline{\rho}_0)\|_{L^2}\leq(B^2_0-\overline{\rho}_0^2)^{\frac{1}{2}}.
$$

(2) {\bf $L^\infty$ estimate}: For the third term of left-hand side of (\ref{eq:3.3}) and (\ref{eq:3.4}), one has
\begin{equation}\label{eq:3.47}
\begin{aligned}
\big|\widetilde{\rho}\Delta K\ast\rho(t,\overline{x}_t)\big|&\leq \|\rho\Delta K\ast\rho\|_{L^\infty}=\|\rho\Delta K\ast(\rho-\overline{\rho})\|_{L^\infty}\\
&=\|\rho\Lambda^{2-\frac{\beta-d}{2}} K\ast \Lambda^{\frac{\beta-d}{2}}(\rho-\overline{\rho})\|_{L^\infty}\\
&\leq C\|\rho\|_{L^\infty}\|\Lambda^{2-\frac{\beta-d}{2}} K\|_{L^1}\|\Lambda^{\frac{\beta-d}{2}}(\rho-\overline{\rho})\|_{L^\infty}\\
&\leq C\|\rho\|_{L^\infty}\|\Lambda^{\frac{\beta-d}{2}}(\rho-\overline{\rho})\|_{L^\infty}.
\end{aligned}
\end{equation}
By Gagliardo-Nirenberg inequality, one has
\begin{equation}\label{eq:3.48}
\|\Lambda^{\frac{\beta-d}{2}}(\rho-\overline{\rho})\|_{L^\infty}\leq C\|\Lambda^{\beta-d}(\rho-\overline{\rho})\|^{1-\theta}_{L^2}\|\rho-\overline{\rho}\|^{\theta}_{L^\infty},
\end{equation}
where
\begin{equation}\label{eq:3.49}
\theta=\frac{2d-\beta}{3d-2\beta}.
\end{equation}
Combining (\ref{eq:3.47}) and (\ref{eq:3.48}), we deduce by (\ref{eq:3.19}) that for any $t\in [0,\tau_1]$, one has
\begin{equation}\label{eq:3.50}
\begin{aligned}
\big|\widetilde{\rho}\Delta K\ast\rho(t,\overline{x}_t)\big|&\leq C\|\rho\|_{L^\infty}\|\Lambda^{\beta-d}(\rho-\overline{\rho})\|^{1-\theta}_{L^2}\|\rho-\overline{\rho}\|^{\theta}_{L^\infty}\\
&\leq C(\Phi(A))^{\frac{1-\theta}{2}}\|\rho\|_{L^\infty}\|\rho-\overline{\rho}\|^{1-\theta}_{L^2}\|\rho-\overline{\rho}\|^{\theta}_{L^\infty}\\
&\leq C( B_0^2-\overline{\rho}_0^2)^{\frac{1-\theta}{2}}C^\theta_\infty(\Phi(A))^{\frac{1-\theta}{2}}\|\rho\|_{L^\infty}\\
& \leq C_2(\Phi(A))^{\frac{1-\theta}{2}}\widetilde{\rho},
\end{aligned}
\end{equation}
where
$$
C_2=C( B_0^2-\overline{\rho}_0^2)^{\frac{1-\theta}{2}}C^\theta_\infty.
$$
Combining (\ref{eq:3.3}) and (\ref{eq:3.50}), we have
\begin{equation}\label{eq:3.51}
\frac{d}{dt}\widetilde{\rho}\leq -\widetilde{\rho}+C_2(\Phi(A))^{\frac{1-\theta}{2}}\widetilde{\rho}.
\end{equation}
If $\rho(t,x)$ satisfies (\ref{eq:3.22}), then we deduce by Lemma \ref{lem:3.5}, Remark \ref{rem:8} and Gronwall's inequality, one has
\begin{equation}\label{eq:3.52}
\begin{aligned}
\widetilde{\rho}(T^\ast)&\leq \widetilde{\rho}(0)\exp\left(\int_{0}^{T^\ast} -1+C_2(\Phi(A))^{\frac{1-\theta}{2}}dt\right)\\
&\leq \widetilde{\rho}(0)e^{-T^\ast}\exp\left(\int_{\Sigma(t)} C_2(\Phi(A))^{\frac{1-\theta}{2}}dt+\int_{[0,T^\ast]\backslash \Sigma(t)} C_2(\Phi(A))^{\frac{1-\theta}{2}}dt\right)\\
&\leq \widetilde{\rho}(0)e^{-T^\ast}e^{C_2\epsilon_0 T^\ast}e^{C_2(\Phi(A))^{\frac{1-\theta}{2}}T^\ast}.
\end{aligned}
\end{equation}
We choose $\epsilon_0$  is small,  $N$ and $A$ are large enough, such that
\begin{equation}\label{eq:3.53}
C_2\epsilon_0+C_2(\Phi(A))^{\frac{1-\theta}{2}}<1.
\end{equation}
Combining (\ref{eq:3.52}) and (\ref{eq:3.53}), to obtain
$$
\widetilde{\rho}(T^\ast)\leq \widetilde{\rho}(0).
$$
Then we have
$$
\|\rho(T^\ast,\cdot)\|_{L^\infty}\leq C_{\infty}.
$$
This completes the proof of Proposition \ref{prop:3.6}.
\end{proof}

\begin{remark}\label{rem:9}
In fact, the assumption (\ref{eq:3.22}) is reasonable in the proof of Lemma \ref{lem:3.5}. Because if there exist a $t_0\in [0,T^\ast]$, such that
$$
\|(\rho-\overline{\rho})(t_0)\|_{L^2}\leq B_1,
$$
according to the Lemma \ref{lem:3.2}, we deduce that for any $t\in [0,\tau_1]$, one has
$$
\|(\rho-\overline{\rho})(t_0+t)\|_{L^2}\leq 2B_1.
$$
By the similar with (\ref{eq:3.50}), one has
\begin{equation}\label{eq:3.54}
\begin{aligned}
\big|\widetilde{\rho}\Delta K\ast\rho(t,\overline{x}_t)\big|&\leq C\|\rho\|_{L^\infty}\|\Lambda^{\beta-d}(\rho-\overline{\rho})\|^{1-\theta}_{L^2}\|\rho-\overline{\rho}\|^{\theta}_{L^\infty}\\
&\leq C(\Phi(A))^{\frac{1-\theta}{2}}\|\rho\|_{L^\infty}\|\rho-\overline{\rho}\|^{1-\theta}_{L^2}\|\rho-\overline{\rho}\|^{\theta}_{L^\infty}\\
&\leq CB_1^{1-\theta}C^\theta_\infty(\Phi(A))^{\frac{1-\theta}{2}}\|\rho\|_{L^\infty}\\
& \leq CB_1^{1-\theta}C^\theta_\infty\|\rho\|_{L^\infty}\\
&\leq C_2B_1^{1-\theta}\widetilde{\rho},
\end{aligned}
\end{equation}
where $\theta$ is defined in (\ref{eq:3.49}) and
$$
C_2= CC^\theta_\infty.
$$
According to the definition $B_1$ in (\ref{eq:3.21}), we choose $N$ is large, such that
\begin{equation}\label{eq:3.55}
C_2B_1^{1-\theta}<\frac{1}{2}.
\end{equation}
Combining (\ref{eq:3.3}), (\ref{eq:3.24}) and (\ref{eq:3.54}), we have
$$
\frac{d}{dt}\widetilde{\rho}\leq -\widetilde{\rho}+C_2B_1^{1-\theta}\widetilde{\rho}<-\frac{1}{2}\widetilde{\rho}.
$$
By Gronwall's inequality, one has for any $t\in [0,\tau_1]$
$$
\widetilde{\rho}(t_0+t)\leq \widetilde{\rho}(t_0)e^{-\frac{t}{2}}.
$$
\end{remark}

\begin{remark}\label{rem:10}
In fact, we choose $\epsilon_0$ and $N$ must satisfy the (\ref{eq:3.46}), (\ref{eq:3.53}) and (\ref{eq:3.55}). Specifically, we choose $N_0>0$, such that $N>N_0$, one has
\begin{equation}\label{eq:3.56}
C_0\epsilon_0+C_0(\Phi(A))^{\frac{1}{4}}\leq C_0\epsilon_0+C_0(2\lambda^{\beta-d}_N)^{\frac{1}{4}}<2,
\end{equation}
we choose $N_1>0$, such that $N>N_1$, one has
\begin{equation}\label{eq:3.57}
C_1\epsilon_0+C_1(\Phi(A))^{\frac{1-\theta}{2}}\leq C_1\epsilon_0+C_1(2\lambda^{\beta-d}_N)^{\frac{1-\theta}{2}}<1.
\end{equation}
And we choose $N_2>0$, such that $N>N_2$, one has
\begin{equation}\label{eq:3.58}
C_2B_1^{1-\theta}=C_2(C\lambda^{\beta-d}_N\|\rho_0-\overline{\rho}_0\|_{L^2})^{1-\theta}<\frac{1}{2}.
\end{equation}
Combining (\ref{eq:3.56}), (\ref{eq:3.57}) and (\ref{eq:3.58}), we choose $N$, such that
$$
N> \max \{N_0, N_1, N_2\}.
$$
\end{remark}

Now, we establish global $L^\infty$ estimate of the solution to Equation (\ref{eq:1.8}) in the case of weakly mixing.

\begin{corollary}[Global $L^\infty$ estimate]\label{cor:3.7}
Let $2\leq\beta<d, d> 2$, $\rho(t,x)$ is the solution of Equation (\ref{eq:1.8}) with initial data $\rho_0(x)$. If $u(x)$ is weakly mixing and the $\rho_0(x)$ satisfies (\ref{eq:3.1}) and (\ref{eq:3.2}). Under the Assumption \ref{assm:3.3}, then there exist a positive constant $A_0=A(\beta,\rho_0, d)$, such that for $A\geq A_0$, we have
$$
\|\rho(t,\cdot)\|_{L^\infty}\leq 2C_{\infty},\quad  t\in [0,+\infty].
$$
\end{corollary}

\begin{proof}
In fact, we only to prove that $T^\ast=\infty$. If $T^\ast<\infty$, according to local estimate and Assumption \ref{assm:3.3}, we know that the time set is nonempty close set for satisfying (\ref{eq:3.12}) and (\ref{eq:3.13}). And we deduce by Assumption \ref{assm:3.3}, Local estimate and Proposition \ref{prop:3.6} that for any $t\in [0,T^\ast+\tau_1]$, one has
$$
\|\rho(t,\cdot)\|_{L^\infty}\leq 2C_{\infty},\ \ \ \ \|\rho(t,\cdot)-\overline{\rho}\|_{L^2}\leq 2(B_0^2-\overline{\rho}_0^2)^{\frac{1}{2}}.
$$
Then $T^\ast$ is a inner point of time set for satisfying (\ref{eq:3.12}) and (\ref{eq:3.13}). Thus time set for satisfying (\ref{eq:3.12}) and (\ref{eq:3.13}) is a open set. Above all, we know that time set for satisfying (\ref{eq:3.12}) and (\ref{eq:3.13}) is a nonempty close and open subset in $\mathbb{R}^+$. this is impressible if $T^\ast<\infty$. Thus $T^\ast=\infty$. This completes the proof of Corollary \ref{cor:3.7}.
\end{proof}

\vskip .2in

\section{The $W^{3,\infty}$ estimate of solution}

In this section, we establish  the $W^{3,\infty}$ estimate of solution to Equation (\ref{eq:1.8}).

\begin{proposition}\label{prop:4.1}
Let $2\leq\beta<d, d> 2$, $\rho(t,x)$ is the solution of Equation (\ref{eq:1.8}) with initial data $\rho_0(x)$. If $u(x)$ is weakly mixing and the $\rho_0(x)$ satisfies (\ref{eq:3.1}) and (\ref{eq:3.2}). There exists a positive constant $A_0=A(\beta,\rho_0, d)$, if $A\geq A_0$,  then for the solution $\rho(t,x)$ of Equation (\ref{eq:1.8}) and any $T>0$, there exist a positive constant $C=C(A,\rho_0,T)<\infty$, such that
$$
\|\rho\|_{L^\infty (0, T; W^{3,\infty}(\mathbb{T}^d))}\leq C.
$$
\end{proposition}

\begin{proof}
First, we establish the estimate of $\|D\rho\|_{L^\infty}$. Applying the operator $D$ to the Equation (\ref{eq:1.8}), one has
\begin{equation}\label{eq:4.1}
\begin{aligned}
\partial_tD\rho+& Au\cdot\nabla(D\rho)+ADu\cdot\nabla\rho+D\rho
+\nabla(D\rho)\cdot\nabla K \ast \rho\\
&+\nabla\rho\cdot\nabla DK \ast \rho+D\rho\Delta K\ast \rho+\rho\Delta K \ast D\rho=0,
\end{aligned}
\end{equation}
and multiplying both sides of (\ref{eq:4.1}) by $D\rho$, to obtain
\begin{equation}\label{eq:4.2}
\begin{aligned}
\frac{1}{2}\partial_t|D\rho|^2&+\frac{1}{2}Au\cdot\nabla(|D\rho|^2)+ADu\cdot\nabla \rho D\rho
+|D\rho|^2\\
&+\frac{1}{2}\nabla(|D\rho|^2)\cdot\nabla K \ast \rho+\nabla\rho\cdot\nabla D K \ast \rho D\rho\\
&+D\rho \Delta K\ast \rho D\rho+\rho\Delta K \ast D\rho D\rho=0.
\end{aligned}
\end{equation}
The third term of the left-hand side of (\ref{eq:4.2}) can be estimated as
\begin{equation}\label{eq:4.3}
\big|A Du\cdot\nabla \rho D\rho\big|\leq A\|Du\|_{L^\infty}\|D\rho\|^2_{L^\infty}\leq CA\|D\rho\|^2_{L^\infty}.
\end{equation}
For the sixth term, seventh term and eighth term of the left-hand side of (\ref{eq:4.2}), one has
\begin{equation}\label{eq:4.4}
\big|\nabla\rho\cdot\nabla D K \ast \rho D\rho\big|\leq C\|D\rho\|^2_{L^\infty}\|\Delta K\|_{L^1}\|\rho \|_{L^\infty},
\end{equation}
\begin{equation}\label{eq:4.5}
\big|D\rho \Delta K\ast \rho D\rho\big|\leq C\|D\rho\|^2_{L^\infty}\|\Delta K\|_{L^1}\|\rho \|_{L^\infty},
\end{equation}
and
\begin{equation}\label{eq:4.6}
\big|\rho\Delta K \ast D\rho D\rho\big|\leq C\|D\rho\|^2_{L^\infty}\|\Delta K\|_{L^1}\|\rho \|_{L^\infty}.
\end{equation}
If define
$$
\widetilde{|D\rho|}=|D\rho|(t,\overline{x}_{1,t})=\max_{x\in \mathbb{T}^d}|D\rho|(t,x).
$$
Using the vanishing of the derivation at the point maximum, the second term and fifth term of the left-hand side of (\ref{eq:4.2}) can be estimated as
\begin{equation}\label{eq:4.7}
\frac{1}{2}Au\cdot\nabla(|D\rho|^2)(t,\overline{x}_{1,t})=0,\ \ \
\frac{1}{2}\nabla(|D\rho|^2)\cdot\nabla K \ast \rho(t,\overline{x}_{1,t})=0.
\end{equation}
Combining Corollary \ref{cor:3.7}, (\ref{eq:4.2})-(\ref{eq:4.7}) and $2\leq \beta<d$, we deduce that the evolution of $|\widetilde{D\rho}|$ follows
\begin{equation}\label{eq:4.8}
\frac{d}{dt}|\widetilde{D\rho}|^2\leq -2|\widetilde{D\rho}|^2+CA|\widetilde{D\rho}|^2+CC_\infty|\widetilde{D\rho}|^2.
\end{equation}
If $A$ is large enough, the (\ref{eq:4.8}) can be written as
\begin{equation}\label{eq:4.9}
\frac{d}{dt}|\widetilde{D\rho}|^2\leq CA|\widetilde{D\rho}|^2.
\end{equation}
Then we deduce by (\ref{eq:4.9}) that for any $t\geq 0$, one has
$$
|\widetilde{D\rho}(t)|^2\leq e^{CAt}|\widetilde{D\rho_0}|^2.
$$
Namely,
\begin{equation}\label{eq:4.10}
\|D\rho(t, \cdot)\|_{L^\infty}\leq e^{CAt}\|D\rho_0\|_{L^\infty}.
\end{equation}
By the same argument with (\ref{eq:4.10}), we deduce that for any $T>0$, there exist a positive constant $C=C(A,T,\rho_0)$, such that
$$
\|D^2\rho(t, \cdot)\|_{L^\infty}\leq C,\ \ \ \ \ \|D^3\rho(t, \cdot)\|_{L^\infty}\leq C.
$$
This completes the proof of Proposition \ref{prop:4.1}.
\end{proof}

\vskip .2in

\appendix
\section{The proof of Lemma \ref{lem:2.3}}

we introduce enhanced dissipation effect of mixing via resolvent estimate (see \cite{Wei.2021}).
Let $\mathcal{H}$ be a Hilbert space, we denote by $\|\cdot\|$ the norm and by $\langle,\rangle$ the inner product. Let $H$ be a linear operator in $\mathcal{H}$ with the domain $D(H)$, it is defined as follows (see \cite{Kato.1995, Wei.2021})

\begin{define}\label{def:A.1}
A closed operator $H$ in a Hilbert space $\mathcal{H}$ is called m-accretive if the left open half-plane is contained in the resolvent set $\rho(H)$ with
$$
(H+\lambda I)^{-1}\in \mathcal{B}(X),\ \ \  \|(H+\lambda I)^{-1}\|\leq (Re \lambda)^{-1},\ \ \  Re \lambda>0,
$$
where $X$ is defined  in (\ref{eq:2.1}), $\mathcal{B}(X)$ denotes the set of bounded linear operators on $X$ with operator $\|\cdot\|$ and $I$ is the identity operator.
\end{define}

\begin{remark}\label{rem:11}
An m-accretive operator $H$ is  accretive and densely defined 
 namely, $D(H)$ is dense in $X$ and $Re \langle Hf,f\rangle\geq0$ for $f\in D(H)$.
\end{remark}

We denote $e^{-tH}$ is a  semigroup with $-H$ as generator and define
\begin{equation}\label{eq:A.1}
\Psi(H)=\inf\{\|(H-i\lambda)f\|: f\in D(H), \lambda\in \mathbb{R}, \|f\|=1 \}.
\end{equation}

The following result is the Gearchart-Pr\"{u}ss type theorem for accretive operators (see \cite{Bedrossian.201701,Wei.2021}).

\begin{lemma}\label{lem:A.2}
Let $H$ be an m-accretive operator in a Hilbert space $\mathcal{H}$. Then for any $t\geq 0$, we have
$$
\|e^{-tH}\|_{L^2\rightarrow L^2}\leq e^{-t\Psi(H)+\frac{\pi}{2}}.
$$
\end{lemma}

We consider advection diffusion equation

\begin{equation}\label{eq:A.2}
\partial_t\eta-\Delta\eta+Au\cdot \nabla \eta=0,\ \ \  \eta(0,x)=\rho_0(x), \ \ x\in \mathbb{T}^d, \ \ t\geq0,
\end{equation}
where $u=u(x)$ is weakly mixing. Let us denote

\begin{equation}\label{eq:A.3}
H_A=-\Delta+Au\cdot \nabla,
\end{equation}
where
$$
D(H_A)=H^{2}\cap X.
$$
If $\rho_0\in D(H_A)$, then the solution of  Equation (\ref{eq:A.2}) can be expressed by semigroup method, namely
$$
\eta(t,x)=e^{-tH_A}\rho_0(x).
$$

\begin{remark}\label{rem:12}
According to the Definition \ref{def:A.1}, the $H_A$ in (\ref{eq:A.3}) is a m-accretive operator.
\end{remark}

If we define the $\Psi(H_A)$ by (\ref{eq:A.1}) and (\ref{eq:A.3}), the similar result with  Lemma \ref{lem:A.2} is as follows
\begin{lemma}\label{lem:A.3}
Let $H_A$ be defined in (\ref{eq:A.3}), then for any $t\geq 0$, we have
$$
\|e^{-tH_A}\|_{L^2\rightarrow L^2}\leq e^{-t\Psi(H_A)+\frac{\pi}{2}}.
$$
\end{lemma}

If $u$ is weakly mixing, the estimate of $\Psi(H_A)$ is as follows

\begin{lemma}\label{lem:A.4}
Let $u$ is weakly mixing, then
$$
\lim_{A\rightarrow +\infty}\Psi(H_A)=+\infty,
$$
where $\Psi(H_A)$ be defined by (\ref{eq:A.1}) and (\ref{eq:A.3}).
\end{lemma}

\begin{remark}\label{rem:13}
The details of Lemma \ref{lem:A.3} and Lemma \ref{lem:A.3} can see \cite{Wei.2021}.
\end{remark}

\begin{proof}[The proof of Lemma \ref{lem:2.3}]
Combining Lemma \ref{lem:A.3} and Lemma \ref{lem:A.4}, we deduce that for $A$ is large enough and any $t\geq 0$, there exist a positive constant $C>1$, such that for any $f\in \mathcal{S}(\mathbb{T}^d)$, one has
\begin{equation}\label{eq:A.4}
\|P_Ne^{-tH_A}f\|_{L^2}\leq C\|P_Nf\|_{L^2}.
\end{equation}
Since
\begin{equation}\label{eq:A.5}
\|P_Ne^{-tH_A}f\|^2_{L^2}=\sum_{|k|\leq N}e^{-2\lambda_k t}\left|\widehat{e^{-tAu\cdot \nabla}f}(k) \right|^2,
\end{equation}
then we deduce by (\ref{eq:A.4}) and (\ref{eq:A.5}) that
$$
\|P_Ne^{-tAu\cdot \nabla}f\|^2_{L^2}=\sum_{|k|\leq N}\left|\widehat{e^{-tAu\cdot \nabla}f}(k) \right|^2\leq Ce^{2\lambda_N t}\|P_Nf\|_{L^2}.
$$
Thus
$$
\|P_Ne^{-tAu\cdot \nabla}f\|_{L^2}\leq Ce^{N^2 t}\|P_Nf\|_{L^2}.
$$
This completes the proof of Lemma \ref{lem:2.3}.
\end{proof}

\vskip .1in
\noindent \textbf{Acknowledgement.} The work of the first author was partially supported by the National Natural Science Foundation of China (Grant No. 11771284) and the Natural Science Foundation Project of Shanghai 2021 "science and technology innovation action plan" (Grant No. 21JC1403600). The work of the second author was partially supported by the National Natural Science Foundation of China (Grant No. 11771284, Grant No. 11831011) and the Natural Science Foundation Project of Shanghai 2021 "science and technology innovation action plan" (Grant No. 21JC1403600).



\bibliographystyle{abbrv}
\bibliography{bib}

\begin{thebibliography}{10}

\bibitem{Ascasibar.2013}
Y.~Ascasibar, R.~Granero-Belinch\'{o}n, and J.~M. Moreno.
\newblock An approximate treatment of gravitational collapse.
\newblock {\em Phys. D}, 262:71--82, 2013.

\bibitem{Bedrossian.201701}
J.~Bedrossian and M.~Coti~Zelati.
\newblock Enhanced dissipation, hypoellipticity, and anomalous small noise
  inviscid limits in shear flows.
\newblock {\em Arch. Ration. Mech. Anal.}, 224(3):1161--1204, 2017.

\bibitem{Bedrossian.2017}
J.~Bedrossian and S.~He.
\newblock Suppression of blow-up in {P}atlak-{K}eller-{S}egel via shear flows.
\newblock {\em SIAM J. Math. Anal.}, 49(6):4722--4766, 2017.

\bibitem{Bedrossian.201501}
J.~Bedrossian and N.~Masmoudi.
\newblock Inviscid damping and the asymptotic stability of planar shear flows
  in the 2{D} {E}uler equations.
\newblock {\em Publ. Math. Inst. Hautes \'{E}tudes Sci.}, 122:195--300, 2015.

\bibitem{Bedrossian.201602}
J.~Bedrossian, N.~Masmoudi, and V.~Vicol.
\newblock Enhanced dissipation and inviscid damping in the inviscid limit of
  the {N}avier-{S}tokes equations near the two dimensional {C}ouette flow.
\newblock {\em Arch. Ration. Mech. Anal.}, 219(3):1087--1159, 2016.

\bibitem{Biler.2010}
P.~Biler and G.~Karch.
\newblock Blowup of solutions to generalized {K}eller-{S}egel model.
\newblock {\em J. Evol. Equ.}, 10(2):247--262, 2010.

\bibitem{Blanchet.2006}
A.~Blanchet, J.~Dolbeault, and B.~Perthame.
\newblock Two-dimensional {K}eller-{S}egel model: optimal critical mass and
  qualitative properties of the solutions.
\newblock {\em Electron. J. Differential Equations}, pages No. 44, 32, 2006.

\bibitem{Chen.2016}
J.~Chen, Y.~Li, and W.~Wang.
\newblock Global classical solutions to the {C}auchy problem of conservation
  laws with degenerate diffusion.
\newblock {\em J. Differential Equations}, 260(5):4657--4682, 2016.

\bibitem{Constantin.2008}
P.~Constantin, A.~Kiselev, L.~Ryzhik, and A.~Zlato\v{s}.
\newblock Diffusion and mixing in fluid flow.
\newblock {\em Ann. of Math. (2)}, 168(2):643--674, 2008.

\bibitem{Corrias.2004}
L.~Corrias, B.~Perthame, and H.~Zaag.
\newblock Global solutions of some chemotaxis and angiogenesis systems in high
  space dimensions.
\newblock {\em Milan J. Math.}, 72:1--28, 2004.

\bibitem{Michele.2020}
M.~Coti~Zelati, T.~M. Elgindi, and K.~Widmayer.
\newblock Enhanced dissipation in the {N}avier-{S}tokes equations near the
  {P}oiseuille flow.
\newblock {\em Comm. Math. Phys.}, 378(2):987--1010, 2020.

\bibitem{Cycon.1987}
H.~L. Cycon, R.~G. Froese, W.~Kirsch, and B.~Simon.
\newblock {\em Schr\"{o}dinger operators with application to quantum mechanics
  and global geometry}.
\newblock Texts and Monographs in Physics. Springer-Verlag, Berlin, study
  edition, 1987.

\bibitem{Evans.2010}
L.~C. Evans.
\newblock {\em Partial differential equations}, volume~19 of {\em Graduate
  Studies in Mathematics}.
\newblock American Mathematical Society, Providence, RI, second edition, 2010.

\bibitem{Fayad.2006}
B.~Fayad.
\newblock Smooth mixing flows with purely singular spectra.
\newblock {\em Duke Math. J.}, 132(2):371--391, 2006.

\bibitem{Fayad.2002}
B.~R. Fayad.
\newblock Weak mixing for reparameterized linear flows on the torus.
\newblock {\em Ergodic Theory Dynam. Systems}, 22(1):187--201, 2002.

\bibitem{Friedman.1969}
A.~Friedman.
\newblock {\em Partial differential equations}.
\newblock Holt, Rinehart and Winston, Inc., New York-Montreal, Que.-London,
  1969.

\bibitem{Hopf.2018}
K.~Hopf and J.~L. Rodrigo.
\newblock Aggregation equations with fractional diffusion: preventing
  concentration by mixing.
\newblock {\em Commun. Math. Sci.}, 16(2):333--361, 2018.

\bibitem{Ionescu.2020}
A.~D. Ionescu and H.~Jia.
\newblock Inviscid damping near the {C}ouette flow in a channel.
\newblock {\em Comm. Math. Phys.}, 374(3):2015--2096, 2020.

\bibitem{Jia.2020}
H.~Jia.
\newblock Linear inviscid damping in {G}evrey spaces.
\newblock {\em Arch. Ration. Mech. Anal.}, 235(2):1327--1355, 2020.

\bibitem{Kato.1995}
T.~Kato.
\newblock {\em Perturbation theory for linear operators}.
\newblock Classics in Mathematics. Springer-Verlag, Berlin, 1995.
\newblock Reprint of the 1980 edition.

\bibitem{Kiselev.2016}
A.~Kiselev and X.~Xu.
\newblock Suppression of chemotactic explosion by mixing.
\newblock {\em Arch. Ration. Mech. Anal.}, 222(2):1077--1112, 2016.

\bibitem{Lafleche.2019}
L.~Lafleche and S.~Salem.
\newblock Fractional {K}eller-{S}egel equation: global well-posedness and
  finite time blow-up.
\newblock {\em Commun. Math. Sci.}, 17(8):2055--2087, 2019.

\bibitem{Li.2010}
D.~Li, J.~L. Rodrigo, and X.~Zhang.
\newblock Exploding solutions for a nonlocal quadratic evolution problem.
\newblock {\em Rev. Mat. Iberoam.}, 26(1):295--332, 2010.

\bibitem{Lin.2019}
Z.~Lin and M.~Xu.
\newblock Metastability of {K}olmogorov flows and inviscid damping of shear
  flows.
\newblock {\em Arch. Ration. Mech. Anal.}, 231(3):1811--1852, 2019.

\bibitem{Zeng.2011}
Z.~Lin and C.~Zeng.
\newblock Inviscid dynamical structures near {C}ouette flow.
\newblock {\em Arch. Ration. Mech. Anal.}, 200(3):1075--1097, 2011.

\bibitem{Masmoudi.2020}
N.~Masmoudi and W.~Zhao.
\newblock Enhanced dissipation for the 2{D} {C}ouette flow in critical space.
\newblock {\em Comm. Partial Differential Equations}, 45(12):1682--1701, 2020.

\bibitem{Nagai.1995}
T.~Nagai.
\newblock Blow-up of radially symmetric solutions to a chemotaxis system.
\newblock {\em Adv. Math. Sci. Appl.}, 5(2):581--601, 1995.

\bibitem{Senba.2002}
T.~Senba and T.~Suzuki.
\newblock Weak solutions to a parabolic-elliptic system of chemotaxis.
\newblock {\em J. Funct. Anal.}, 191(1):17--51, 2002.

\bibitem{Shi.2021}
B.~Shi.
\newblock Suppression of blow up by mixing in generalized keller-segel system
  with fractional dissipation and strong singular kernel.
\newblock {\em Submitted arXiv:2103.04484}.

\bibitem{Shi.2019}
B.~Shi and W.~Wang.
\newblock Suppression of blow up by mixing in generalized {K}eller-{S}egel
  system with fractional dissipation.
\newblock {\em Commun. Math. Sci.}, 18(5):1413--1440, 2020.

\bibitem{Tao.2006}
T.~Tao.
\newblock {\em Nonlinear dispersive equations}, volume 106 of {\em CBMS
  Regional Conference Series in Mathematics}.
\newblock Published for the Conference Board of the Mathematical Sciences,
  Washington, DC; by the American Mathematical Society, Providence, RI, 2006.
\newblock Local and global analysis.

\bibitem{Wang.2001}
W.~Wang and T.~Yang.
\newblock The pointwise estimates of solutions for {E}uler equations with
  damping in multi-dimensions.
\newblock {\em J. Differential Equations}, 173(2):410--450, 2001.

\bibitem{Wei.2021}
D.~Wei.
\newblock Diffusion and mixing in fluid flow via the resolvent estimate.
\newblock {\em Sci. China Math.}, 64(3):507--518, 2021.

\bibitem{Wei.201901}
D.~Wei and Z.~Zhang.
\newblock Enhanced dissipation for the {K}olmogorov flow via the hypocoercivity
  method.
\newblock {\em Sci. China Math.}, 62(6):1219--1232, 2019.

\bibitem{Zhang.2018}
D.~Wei, Z.~Zhang, and W.~Zhao.
\newblock Linear inviscid damping for a class of monotone shear flow in
  {S}obolev spaces.
\newblock {\em Comm. Pure Appl. Math.}, 71(4):617--687, 2018.

\bibitem{Zhang.201902}
D.~Wei, Z.~Zhang, and W.~Zhao.
\newblock Linear inviscid damping and vorticity depletion for shear flows.
\newblock {\em Ann. PDE}, 5(1):Paper No. 3, 101, 2019.

\bibitem{Wei.2020}
D.~Wei, Z.~Zhang, and W.~Zhao.
\newblock Linear inviscid damping and enhanced dissipation for the {K}olmogorov
  flow.
\newblock {\em Adv. Math.}, 362:106963, 103, 2020.

\bibitem{Zi.2021}
L.~Zeng, Z.~Zhang, and R.~Zi.
\newblock Suppression of blow-up in {P}atlak-{K}eller-{S}egel-{N}avier-{S}tokes
  system via the {C}ouette flow.
\newblock {\em J. Funct. Anal.}, 280(10):108967, 2021.

\bibitem{Zillinger.2016}
C.~Zillinger.
\newblock Linear inviscid damping for monotone shear flows in a finite periodic
  channel, boundary effects, blow-up and critical {S}obolev regularity.
\newblock {\em Arch. Ration. Mech. Anal.}, 221(3):1449--1509, 2016.

\bibitem{Zillinger.2017}
C.~Zillinger.
\newblock Linear inviscid damping for monotone shear flows.
\newblock {\em Trans. Amer. Math. Soc.}, 369(12):8799--8855, 2017.

\end{thebibliography}

\end{document}